\documentclass[10pt]{amsart}
\usepackage{latexsym,amssymb,amsfonts,amsmath, amscd}
\usepackage[T1]{fontenc}
\usepackage[english]{babel}
\usepackage[pdftex]{graphicx}
\usepackage{pgf,tikz}
\usetikzlibrary{arrows, arrows.meta}
\usetikzlibrary{positioning,calc}
\usepackage{amsthm, enumerate}
\usepackage{url}
\usepackage[top=3cm, bottom=3cm, left=4cm, right=4cm]{geometry}
\usepackage[noadjust]{cite}

\newcommand{\dist}{\operatorname{dist}}
\newcommand{\diam}{\operatorname{diam}}
\newcommand{\acr}{\newline\indent}
\numberwithin{equation}{section}

\makeatletter

\renewcommand{\p@enumii}{}

\makeatother

\newcommand{\btk}{\begin{tikzpicture}}
\newcommand{\etk}{\end{tikzpicture}}
\newcommand{\bce}{\begin{center}}
\newcommand{\ece}{\end{center}}

\newcommand{\CC}{\ensuremath{\mathbb C}}

\newcommand{\nit}{\ensuremath{\mathbb N}}

\newcommand{\bqn}{\begin{eqnarray*}}
\newcommand{\eqn}{\end{eqnarray*}}
\newcommand{\bqnm}{\begin{matrix}}
\newcommand{\eqnm}{\end{matrix}}

\newtheorem{theorem}{Theorem}[section]
\newtheorem*{theorem*}{Theorem}
\newtheorem{corollary}[theorem]{Corollary}
\newtheorem{corollary*}{Corollary}
\newtheorem{lemma}[theorem]{Lemma}
\newtheorem{proposition}[theorem]{Proposition}

\theoremstyle{definition}
\newtheorem{definition}[theorem]{Definition}
\newtheorem{example}[theorem]{Example}

\theoremstyle{remark}
\newtheorem{remark}[theorem]{Remark}

\newtheorem{mini-prob}{Mini-probl\`eme}

\makeatletter

\renewcommand{\p@enumii}{}

\usepackage{fancyhdr}
\makeatother

\allowdisplaybreaks

\begin{document}

\title{Proximinal sets and connectedness in graphs}

\author{Karim Chaira}
\address{\textbf{Karim Chaira}\acr
CRMEF, Rabat-Sal\'{e}-K\'{e}nitra \acr
Avenue Allal El Fassi, Bab Madinat Al Irfane\acr
B.P 6210, 10000 Rabat, Morocco}
\email{chaira\_karim@yahoo.fr}

\author{Oleksiy Dovgoshey}
\address{\textbf{Oleksiy Dovgoshey}\acr
Department of Theory of Functions \acr
Institute of Applied Mathematics and Mechanics of NASU \acr
84100 Slovyansk, Ukraine and \acr
Department of Mathematics and Statistics \acr
University of Turku \acr
Fin-20014, Turku, Finland}
\email{oleksiy.dovgoshey@gmail.com}



\begin{abstract}
Let \(G\) be a graph with a vertex set \(V\). The graph \(G\) is path-proximinal if there are a semimetric \(d \colon V \times V \to [0, \infty[\) and disjoint proximinal subsets of the semimetric space \((V, d)\) such that \(V = A \cup B\), and vertices \(u\), \(v \in V\) are adjacent iff
\[
d(u, v) \leqslant \inf \{d(x, y) \colon x \in A, y \in B\},
\]
and, for every \(p \in V\), there is a path connecting \(A\) and \(B\) in \(G\), and passing through \(p\). It is shown that a graph is path-proximinal if and only if all its vertices are not isolated. It is also shown that a graph is simultaneously proximinal and path-proximinal for an ultrametric if and only if the degree of every its vertex is equal to \(1\).
\end{abstract}

\maketitle

\section{Introduction and Preliminaries}

\subsection{Introduction}

A bipartite graph $G(A,B)$ with fixed parts $A$ and $B$ is said to be proximinal if there exists a semimetric space $(X,d)$ such that $A$ and $B$ are disjoint proximinal subsets of $X$ and vertices $a\in A$ and $b\in B$ are adjacent if and only if $d(a,b) = \dist(A,B)$. The structure of proximinal bipartite graphs for semimetric and metric spaces was described in \cite{key-6}. In particular, it is proved that a bipartite graph $G$ is not isomorphic to any proximinal graph if and only if $G$ is finite and empty. In \cite{key-9}, the authors characterized the semimetric spaces whose proximinal graphs have at most one edge and the semimetrics spaces whose proximinal graphs have the vertices of degree at most one only. This allows them to find the necessary and sufficient conditions for the uniqueness of the best proximity pairs and the best approximations. Some references concerning graphs and best proximity points are given in \cite{key-4, key-6, key-25, key-26, key-27, key-29, key-30}.

In this paper we continue to study the interaction between proximity and graphs by introducing path-proximinal graphs as graphs that are the union of all paths starting at nearest between \(A\) and \(B\) points and having edges \(\{x, y\}\) which satisfy the inequality \(d(x, y) \leqslant \dist(A, B)\).

\subsection{Semimetrics and proximinal sets}
Let $X$ be a nonvoid set. A \emph{semimetric} on $X$ is a function $d\colon X \times X \rightarrow [0,+\infty[$ such that $d(a,b) = d(b,a)$ and
\[
\bigl(d(a, b) = 0\bigr) \Leftrightarrow \bigl(a= b\bigr)
\]
for all $a, b \in X$. A pair $(X,d)$, where $d$ is a semimetric on $X$, is called a \emph{semimetric space}. A semimetric $d$ is a \emph{metric} if the triangle inequality
\[
d(a,b) \leqslant d(a,c) + d(c,b)
\]
holds for all \(a\), \(b\), \(c \in X\). A semimetric $d$ is an \emph{ultrametric} if we have
\[
d(a,b) \leqslant \max\bigl\{d(a,c),d(c,b)\bigr\}
\]
for all \(a\), \(b\), \(c \in X\). Every ultrametric space is a metric space and every metric space is a semimetric space.

\begin{definition}\label{d2.1}
Let \((X, d) \) a semimetric space, let \(A\) be a nonempty subset of \(X\) and let \(x \in X\). The point \(a_0\in A\) is called a \emph{best approximation} to \(x\) (in \(A\)) if
\begin{equation}\label{d1.1:e1}
d(x,a_0) = \inf\bigl\{d(x, a)\colon a\in A\bigr\}.
\end{equation}
The set \(A\) is said to be \emph{proximinal} if \(A\) contains a best approximation to every point of~\(X\).
\end{definition}

\begin{remark} In \cite{key-30} Ivan Singer wrote: ``The term «proximinal» set (a combination of «proximity» and «minimal») was proposed by R. Killgrove and used first by R.~R.~Phelps \cite{key-26}.''
\end{remark}

Let \(A\) and \(B\) be subsets of a semimetric space \((X, d) \). We will say that the pair \((A, B)\) is \emph{proximinal} if \(A\) and \(B\) are proximinal in \((X, d)\).

For nonempty subsets \(A\) and \(B\) of a semimetric space \((X,d)\), we define a distance from \(A\) to \(B\) as
\begin{equation}\label{e1.2}
\dist(A,B) := \inf\{d(a,b)\colon a\in A\ \text{and}\ b\in B\}.
\end{equation}
If \(A\) is a one-point set, \(A = \{a\}\), then, for brevity, we write \(\dist(a, B)\) instead of \(\dist(\{a\}, B)\).

\begin{definition}\label{d2.2}
Let \((X, d)\) be a semimetric space, and let \(A\), \(B\) be nonempty subsets of \(X\). Write
\begin{gather}\label{d1.2:e1}
A_0 :=\{a\in A\colon d(a, b) = \dist(A, B)\ \text{for some}\ b\in B\},\\
\label{d1.2:e2}
B_0 := \{b\in B\colon d(a, b) = \dist(A, B)\ \text{for some}\ a\in A\}.
\end{gather}
A pair \((a_0, b_0) \in A_0 \times B_0\) is called a \emph{best proximity pair} for \(A\) and \(B\) if
$$
d(a_0, b_0) = \dist(A, B).
$$
\end{definition}

\begin{remark}\label{r2.1}
For every pair  \(A\), \(B\) of nonempty subsets of a semimetric space \((X, d) \), we have \(A_0 \neq \varnothing\) if and only if \(B_0 \neq \varnothing\). If \(S \subseteq A \times B\) is a set of all best proximity pairs for \(A\) and \(B\), then \(A_0\) (\(B_0\)) is the projection of \(S\) on \(A\) (\(B\)).
\end{remark}

The next result is a part of Theorem 2.6 from \cite{key-5}.
\begin{theorem}\label{t2.1}
Let $(A,B)$ be a proximinal pair in an ultrametric space $(X,d)$. Then the following statements are equivalent:
\begin{enumerate}
\item The inequality $\diam(B) \leqslant \dist(A,B)$ holds.
\item The set $A_{0} \subseteq A$ is a proximinal subset of $X$, and the equality $B_{0} = B$ holds, and every $(a,b) \in A_{0}\times B_{0}$ is a best proximity pair for the sets $A$ and $B$.
\end{enumerate}
\end{theorem}

\subsection{Graphs}
A \emph{simple graph} is a pair \((V, E)\) consisting of a nonempty set \(V\) and a set \(E\) whose elements are unordered pairs of different elements of \(V\). In what follows, we will consider the simple graphs only.

For a graph \(G = (V, E)\), the sets \(V = V (G)\) and \(E = E(G)\) are called the \emph{set of vertices} and the \emph{set of edges}, respectively. Two vertices \(u\), \(v \in V\) are \emph{adjacent} if \(\{u, v\}  \in E(G)\). A \emph{complete} graph is a graph in which every two different vertices are adjacent. A vertex \(v \in V(G)\) is \emph{isolated} if there are no vertices which are adjacent with \(v\) in \(G\). We say that \(G\) is \emph{empty} if \(E(G) = \varnothing\). Thus, \(G\) is empty iff all vertices of \(G\) are isolated.

A graph \(H\) is a \emph{subgraph} of a graph \(G\) if \(V (H) \subseteq V (G) \) and \(E(H) \subseteq E(G)\) are valid.

If \(G\) is a nonempty graph, then we will denote by \(G'\) a subgraph of \(G\) whose vertices are non-isolated vertices of \(G\) and such that \(E(G') = E(G)\). It is easy to see that \(V(G')\) is the union of all two-point sets \(\{a, b\} \in E(G)\).

\begin{remark}\label{r2.2}
The graph \(G'\) can be characterized by the following extremal property: If \(H\) is a subgraph of \(G\) such that \(G' \subseteq H\) holds and \(H\) does not have any isolated vertices, then \(G' = H\).
\end{remark}
A graph \(G\) is \emph{finite} if \(V (G)\) is a finite set, \(|V (G)| < \infty\). Following \cite{BM2008GT} we define a \emph{path} as a finite nonempty graph \(P\) whose vertices can be numbered so that
\[
V(P) = \{u_0, u_1, \ldots, u_k\}, \quad k \geqslant 1, \quad \text{and} \quad E(P) = \{\{u_0, u_1\}, \ldots, \{u_{k-1}, u_k\}\}.
\]
In this case we say that \(P\) is a path joining \(u_0\) and \(u_k\), and write \(P = (u_0, u_1, \ldots, u_k)\). A graph \(G\) is \emph{connected} if, for every two distinct \(u\), \(v \in V (G)\), there is a path \(P \subseteq G\) joining \(u\) and \(v\).

Let $\mathcal{F}$ be a nonempty set of graphs. A graph $H$ is called  the union of graphs $G\in \mathcal{F}$ if
$$
V(H)= \underset{G\in \mathcal{F}}{\cup} V(G) \text{ and }
E(H)= \underset{G\in \mathcal{F}}{\cup} E(G).
$$
We say 	that the union $H$ is \emph{disjoint} if $V(G_{1}) \cap V(G_{2}) = \varnothing$ holds for all different graphs $G_{1}$, $G_{2}\in \mathcal{F}$.

A subgraph $H$ of a graph $G$ is a \emph{connected component} of $G$ if the implication
\begin{equation}\label{e2.5}
(H \subseteq \Gamma) \Rightarrow (H=\Gamma)
\end{equation}
is valid for every connected graph $\Gamma \subseteq G$.

In the next section of the paper we will use the following simple lemmas describing some properties of connected subgraphs.

\begin{lemma}\label{l2.8}
Every graph is the disjoint union of its connected components.
\end{lemma}

\begin{lemma}\label{l2.9}
Let \(G_1\) and \(G_2\) be connected graphs. If \(V(G_1) \cap V(G_2) \neq \varnothing\), then the union \(G_1 \cup G_2\) is also connected.
\end{lemma}

\begin{lemma}\label{l2.10}
Let \(H\) be a graph and let \(H_1\) be a connected subgraph of \(H\) such that \(V(H_1) = V(H)\). Then \(H\) is connected.
\end{lemma}

\begin{lemma}\label{l2.11}
Let \(W_a\) be a connected component of a graph \(W\), \(a \in V(W_a)\) and let \(W^*\) be a connected subgraph of \(W\). If \(a \in V(W^*)\) is valid, then \(W^*\) is a subgraph of \(W_a\).
\end{lemma}

Proofs of Lemmas~\ref{l2.8}--\ref{l2.10} are simple and we omit it here. Lemma~\ref{l2.11} follows from Lemma~\ref{l2.9} and \eqref{e2.5}.

Let \(S\) be an arbitrary nonvoid set of vertices of a graph \(G\). The \emph{subgraph induced in \(G\) by \(S\)} is a graph \(G[S]\) such that \(V(G[S]) = S\) and, for all \(u\), \(v \in S\), we have \(\{u, v\} \in E(G[S])\) iff \(\{u, v\} \in E(G)\).

\begin{definition}\label{d2.3}
A graph \(G\) is \emph{bipartite} if the vertex set \(V(G)\) can be partitioned into two nonvoid disjoint subsets, or \emph{parts}, in such a way that no edge has both ends in the same part.
\end{definition}

Let $G$ be a bipartite graph with parts $A$ and $B$. Then we say that $G$ is \emph{complete bipartite} $G$ if $\{a,b\} \in E(G)$ whenever $a\in A$ and $b\in B$.

By analogy with the concept of induced graphs, one can introduce the concept of induced-bipartite ones.

\begin{definition}\label{d2.14}
Let \(G\) be a graph and let \(A\), \(B\) be disjoint nonempty subsets of \(V(G)\). The \emph{induced-bipartite subgraph} \(G[A, B]\) of \(G\) is the graph whose vertex set is \(A \cup B\) and whose edge set consists of all \(\{u, v\} \in E(G)\) that satisfies \(\{u, v\} \cap A \neq \varnothing \neq \{u, v\} \cap B\).
\end{definition}

The next lemma follows from the definitions of induced graphs and induced-bipartite graphs.

\begin{lemma}\label{l2.15}
Let \(G\) be a graph and let \(A\), \(B\) be disjoint nonempty subsets of \(V(G)\) such that \(A \cup B = V(G)\). Then \(G\) is the union of the induced graphs \(G[A]\), \(G[B]\) and the induced-bipartite graph \(G[A, B]\).
\end{lemma}

\begin{definition}[\cite{key-6}]\label{d2.4}
A bipartite graph \(G = G(A, B)\) with fixed parts \(A\) and \(B\) is \emph{proximinal} for a semimetric space \((X, d) \) if \(A\) and \(B\) are disjoint proximinal subsets of \(X\), and the equivalence
\begin{equation}\label{d1.5:e1}
\bigl(\{a, b\} \in E(G)\bigr) \Leftrightarrow \bigl(d(a, b) = \dist(A, B)\bigr)
\end{equation}
is valid for all \(a \in A\) and \(b \in B\).
\end{definition}

\begin{theorem}[\cite{key-6}]\label{t2.2}
Let $G$ be a bipartite graph with some fixed parts $A$ and $B$. Then the following statements are equivalent:
\begin{enumerate}
\item Either $G$ is nonempty or $G$ is empty but $A$ and $B$ are infinite.
\item $G$ is proximinal for a metric space.
\item $G$ is proximinal for a semimetric space.
\end{enumerate}
\end{theorem}

The main objects of our studies are path-bipartite graphs and path-proximinal graphs which can be defined as follows.

\begin{definition}\label{d3.6}
Let $A$ and $B$ be two nonvoid disjoint sets. A path $P$ is a \emph{be-path} of $A$ and $B$ if $V(P) \subseteq A\cup B$ and there is a unique $\{a_{0},b_{0}\} \in E(P)$ for which
\[
A\cap \{a_{0},b_{0}\} \neq \varnothing \neq B\cap \{a_{0},b_{0}\}.
\]
The union of nonempty set of be-paths of fixed $A$ and $B$ will be called a \emph{path-bipartite graph} of $A$ and $B$. We will say that a graph \(G\) is \emph{path-bipartite} if there are \(A\), \(B \subseteq V(G)\) such that \(G\) is path-bipartite of \(A\) and \(B\)
\end{definition}

\begin{definition}\label{d3.7}
Let $G$ be a path-bipartite graph of sets $A$ and $B$, let $X := A\cup B$ and let $d\colon X\times X \rightarrow [0,+\infty[$ be a semimetric. The graph $G$ is \emph{path-proximinal} for \(A\) and \(B\) w.r.t. the semimetric \(d\) if $A$, $B$ are proximinal subsets of $(X,d)$ and the equivalence
\begin{equation}\label{d3.10:e1}
(\{x,y\} \in E(G)) \Leftrightarrow (d(x,y) \leqslant \dist(A,B))
\end{equation}
is valid for all distinct $x$, $y \in X$.

We will say that a graph $G$ is \emph{path-proximinal} if there are a semimetric $d$ on $X := V(G)$ and disjoint nonempty $A$, $B \subseteq V(G)$ such that $G$ is a path-proximinal graph for \(A\) and \(B\) w.r.t. \(d\).
\end{definition}

The paper is organized as follows.

Theorem~\ref{t3.4} describes structure of path-bipartite graph \(G\) for which given points \(a_1 \in A\) and \(b_1 \in B\) can be joined by be-path in \(G\). In Theorem~\ref{t3.6} we consider a bipartite graph \(\mathbf{G}(\mathbf{A}, \mathbf{B})\) corresponding to path-bipartite graph \(G(A, B)\) and show that \(\mathbf{G}(\mathbf{A}, \mathbf{B})\) is complete bipartite iff any \(a \in A\) and \(b \in B\) can be joining by be-path in \(G(A, B)\). Corollary~\ref{c2.9} describes the path-bipartite graphs \(G(A, B)\) which are connected iff the corresponding graphs \(\mathbf{G}(\mathbf{A}, \mathbf{B})\) are complete-bipartite. Theorem~\ref{t3.9} give us necessary and sufficient conditions under which \(G\) is path-bipartite for fixed \(A\), \(B \subseteq V(G)\). It is the one of the main results of Section~\ref{sec2}. Using this theorem we characterize the path-bipartite graphs up to isomorphism in Corollaries \ref{c3.10}--\ref{c2.13}.

The properties of path-proximinal graph are studied in Section~\ref{sec3}.

Theorem~\ref{t3.5} describes the structure of semimetric spaces \((X, d)\) with disjoint proximinal subsets \(A\), \(B\) for which graphs defined by \eqref{d3.10:e1} are path-bipartite. Theorem~\ref{t3.7} shows that for every path-bipartite graph \(G\) there is a metric such that \(G\) is path-proximinal with respect to this metric. This result allows us to characterize path-proximinal graphs up to isomorphism in Theorem~\ref{t3.16}. In Propositions~\ref{p3.22} and \ref{p3.9} we describe the graphs which are proximinal and path-proximinal simultaneously. The final result, Theorem~\ref{t3.10}, shows that a proximinal graph \(G\) is path-proximinal w.r.t. an ultrametric iff every vertex of \(G\) has degree \(1\).

\section{Path-bipartite graphs}
\label{sec2}

Below we will consider the path-bipartite graphs of arbitrary cardinality.

\begin{proposition}\label{p3.2}
Let \(G\) be a path-bipartite graph of \(A\) and \(B\). Then \(V(G) = A \cup B\) holds.
\end{proposition}

\begin{proof}
By Definition~\ref{d3.6}, we have
\begin{equation}\label{p3.2:e1}
A \cup B \subseteq V(G).
\end{equation}
Since \(G\) is a path-bipartite graph of \(A\) and \(B\), there is a set \(\mathbf{P}\) of be-paths \(P\) of \(A\) and \(B\) such that
\begin{equation}\label{p3.2:e2}
V(G) \subseteq \bigcup_{P \in \mathbf{P}} V(P).
\end{equation}
For every \(P \in \mathbf{P}\), Definition~\ref{d3.6} implies \(V(P) \subseteq A \cup B\). Hence, \eqref{p3.2:e2} implies the inclusion
\begin{equation}\label{p3.2:e3}
V(G) \subseteq A \cup B.
\end{equation}
Now the equality \(V(G) = A \cup B\) follows from \eqref{p3.2:e1} and \eqref{p3.2:e3}.
\end{proof}

Let $G$ be a path-bipartite graph of $A$ and $B$. Then we write
\begin{equation}\label{e3.2}
\mathcal{B}_{path}(G):= \{(a,b)\in A\times B\colon \text{there is a be-path } P_{A,B}\subseteq G \text{ joining } a \text{ and } b\}.
\end{equation}
If $\mathcal{B}_{path}(G)=A\times B$, then \(G\) is called \emph{path-complete}.

\begin{remark}\label{r3.4}
$\mathcal{B}_{path}(G)$ is a nonempty set for every path-bipartite graph $G$ of \(A\) and \(B\).
\end{remark}

\begin{theorem}\label{t3.4}
Let \(G = G(A, B)\) be a path-bipartite graph of sets \(A\) and \(B\), let \(a_1 \in A\) and \(b_1 \in B\) be given, and let \(G_{a_1}[A]\) and \(G_{b_1}[B]\) be the connected components of the induced graphs \(G[A]\) and \(G[B]\) such that \(a_1 \in A_1\) and \(b_1 \in B_1\), where
\begin{equation}\label{t3.4:e1}
A_1 := V(G_{a_1}[A]) \text{ and } B_1 := V(G_{b_1}[B]).
\end{equation}
Write
\begin{equation}\label{t3.4:e2}
C_1 := A_1 \cup B_1.
\end{equation}
Then the following statements are equivalent:
\begin{enumerate}
\item \label{t3.4:s1} The pair \((a_1, b_1)\) belongs to \(\mathcal{B}_{path}(G)\).
\item \label{t3.4:s2} The subgraph \(G[C_1]\) induced in \(G\) by the set \(C_1\) is connected.
\item \label{t3.4:s3} The inclusion \(A_1 \times B_1 \subseteq \mathcal{B}_{path}(G)\) holds.
\end{enumerate}
\end{theorem}

\begin{proof}
\(\ref{t3.4:s1} \Rightarrow \ref{t3.4:s2}\). Suppose that \ref{t3.4:s1} holds,
\begin{equation}\label{t3.4:e3}
(a_1, b_1) \in \mathcal{B}_{path}(G).
\end{equation}
Then there is a be-path \(P_{A, B}\) joining \(a_1 \in A\) and \(b_1 \in B\). By Lemma~\ref{l2.9}, the union \(P_{A, B} \cup G_{a_1}[A]\) is connected because \(a_1 \in V(P_{A, B}) \cap V(G_{a_1}[A])\) and the graphs \(P_{A, B}\) and \(G_{a_1}[A]\) are connected. Similarly, the union \(P_{A, B} \cup G_{b_1}[B]\) is also connected. Let us consider a graph
\begin{equation}\label{t3.4:e4}
H_1 := P_{A, B} \cup G_{a_1}[A] \cup G_{b_1}[B].
\end{equation}
Using Lemma~\ref{l2.9} again we see that \(H_1\) is connected. By Lemma~\ref{l2.10}, to complete the proof of statement~\ref{t3.4:s2} it suffices to show that
\begin{equation}\label{t3.4:e5}
H_1 \subseteq G[C_1]
\end{equation}
and
\begin{equation}\label{t3.4:e6}
V(H_1) = V(G[C_1]).
\end{equation}

It follows from \eqref{t3.4:e1}, \eqref{t3.4:e2} and \eqref{t3.4:e4} that
\[
V(H_1) \supseteq A_1 \cup B_1 = V(G[C_1]).
\]
Consequently, \eqref{t3.4:e6} holds if we have \eqref{t3.4:e5}. For proof \eqref{t3.4:e5}, we note that \(G_{a_1}[A_1]\) and \(G_{b_1}[B_1]\) are subgraphs of \(G[C_1]\)
\begin{equation}\label{t3.4:e7}
G_{a_1}[A_1] \subseteq G[C_1] \text{ and } G_{b_1}[B_1] \subseteq G[C_1].
\end{equation}
Hence, \eqref{t3.4:e5} holds if
\begin{equation}\label{t3.4:e8}
P_{A, B} \subseteq G[C_1].
\end{equation}
Let us prove the last inclusion.

By Definition~\ref{d3.6}, we can find points \(a_0 \in A\), \(b_0 \in B\) and connected graphs \(P_A \subseteq G[A]\), \(P_B \subseteq G[B]\) such that
\begin{equation}\label{t3.4:e9}
a_0, a_1 \in V(P_A), \quad b_0, b_1 \in V(P_B)
\end{equation}
and
\begin{equation}\label{t3.4:e10}
P_{A, B} = P_A \cup P_B \cup P^{0},
\end{equation}
where \(P^{0}\) is a path defined by
\begin{equation}\label{t3.4:e11}
V(P^{0}) := \{a_0, b_0\}, \quad E(P^{0}) := \{\{a_0, b_0\}\}.
\end{equation}
It follows from \eqref{t3.4:e9}--\eqref{t3.4:e11} that \eqref{t3.4:e8} holds whenever
\begin{equation}\label{t3.4:e12}
P_A \subseteq G[C_1] \text{ and } P_B \subseteq G[C_1].
\end{equation}
Lemma~\ref{l2.11} and \eqref{t3.4:e9} imply
\begin{equation}\label{t3.4:e13}
P_A \subseteq G_{a_1}[A] \text{ and } P_B \subseteq G_{b_1}[B].
\end{equation}
Now \eqref{t3.4:e12} follows from \eqref{t3.4:e13} and \eqref{t3.4:e7}.

\(\ref{t3.4:s2} \Rightarrow \ref{t3.4:s3}\). Let \(G[C_1]\) be a connected graph. Then there are points \(a_0 \in A_1\) and \(b_0 \in B_1\) such that
\[
\{a_0, b_0\} \in E(G[C_1]).
\]
It is clear that the graph \(P^0\) defined by \eqref{t3.4:e11} is a be-path of \(A\) and \(B\) joining \(a_0\) and \(b_0\). Hence, we have \((a_0, b_0) \in \mathcal{B}_{path}(G)\).

Let us consider now an arbitrary \(a \in A_1\) and \(b \in B_1\). If \(a \neq a_0\) and \(b \neq b_0\), then there are a path \(P_a\) in \(A_1\) joining \(a\) and \(a_0\), and a path \(P_b\) in \(B_1\) joining \(b\) and \(b_0\). Then the union \(P_a \cup P^{0} \cup P_b\) is a be-path of \(A\) and \(B\) joining \(a\) and \(b\). For the case when \(a = a_0\) or \(b = b_0\), the desired be-path of \(A\) and \(B\) can be constructed similarly.

\(\ref{t3.4:s3} \Rightarrow \ref{t3.4:s1}\). Let \ref{t3.4:s3} hold. Then \ref{t3.4:s1} follows from \((a_1, b_1) \in A_1 \times B_1\) and \(A_1 \times B_1 \subseteq \mathcal{B}_{path}(G)\).
\end{proof}

\begin{remark}\label{r3.5}
Let a graph $G$ be path-bipartite for \(A\) and \(B\). An element $(a,b)$ of $\mathcal{B}_{path}(G)$ can have several be-paths joining $a$ and $b$ in $G$.
\end{remark}

\begin{remark}\label{r3.7}
Let $G$ be a path-bipartite graph of \(A\) and \(B\). If $G$ is path complete, then $G$ is connected, but not vice versa, in general (see Example~\ref{ex3.7} below).
\end{remark}

\begin{example}\label{ex3.7}
Let \(P = (a_1, b_1, a_2, b_2)\) be a path (see Figure~\ref{fig3}).
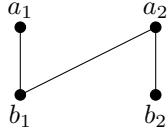
\begin{figure}[htb]
\begin{tikzpicture}
\def\xx{1.8cm}
\def\yy{0.9cm}

\coordinate [label=above:{$a_{1}$}] (a1) at (0*\xx, 0);
\coordinate [label=below:{$b_{1}$}] (b1) at (0*\xx, -\yy);
\coordinate [label=above:{$a_{2}$}] (a2) at (1*\xx, 0);
\coordinate [label=below:{$b_{2}$}] (b2) at (1*\xx, -\yy);

\draw [fill, black] (a1) circle (2pt);
\draw [fill, black] (b1) circle (2pt);
\draw [fill, black] (a2) circle (2pt);
\draw [fill, black] (b2) circle (2pt);
\draw (a1) -- (b1) -- (a2) -- (b2);
\end{tikzpicture}
\caption{The path \(P\) is a path-bipartite graph of sets \(A = \{a_1, a_2\}\) and \(B =  \{b_1, b_2\}\).}
\label{fig3}
\end{figure}
Then \(P\) is a connected path-bipartite graph of the sets \(A = \{a_1, a_2\}\) and \(B =  \{b_1, b_2\}\), but \((a_1, b_2) \notin \mathcal{B}_{path}(P)\).
\end{example}

Let us consider a graph \(G = G(A, B)\) with \(V(G) = A \cup B\), where \(A\) and \(B\) are disjoint nonempty sets, and denote by
\[
\mathbf{A} := \{A^i \colon i \in I\} \quad (\mathbf{B} := \{B^j \colon j \in J\})
\]
the set of all connected components of the induced graph \(G[A]\) (\(G[B]\)).

Now we define a bipartite graph \(\mathbf{G} = \mathbf{G}(\mathbf{A}, \mathbf{B})\) by the rules \(V(\mathbf{G}) := \mathbf{A} \cup \mathbf{B}\) and, for \(i_1 \in I\) and \(j_1 \in J\), \(\{A^{i_1}, B^{j_1}\} \in E(\mathbf{G})\) holds iff there are \(a_1 \in A^{i_1}\) and \(b_1 \in B^{j_1}\) such that
\begin{equation}\label{e3.16}
(a_1, b_1) \in \mathcal{B}_{path}(G).
\end{equation}

\begin{theorem}\label{t3.6}
Let \(G\) be a path-bipartite graph of sets \(A\) and \(B\). Then the following statements are equivalent:
\begin{enumerate}
\item\label{t3.6:s1} \(G\) is path-complete.
\item\label{t3.6:s2} \(\mathbf{G}(\mathbf{A}, \mathbf{B})\) is complete bipartite.
\end{enumerate}
\end{theorem}

\begin{proof}
\(\ref{t3.6:s1} \Rightarrow \ref{t3.6:s2}\). Let \(G\) be path-complete. We must show that \(\mathbf{G}(\mathbf{A}, \mathbf{B})\) is complete bipartite. It is valid iff, for all \(i_1 \in I\) and \(j_1 \in J\), there are \(a_1 \in A^{i_1}\) and \(b_1 \in B^{j_1}\) such that \eqref{e3.16} holds. Since \(G\) is path-complete, \((a, b) \in \mathcal{B}_{path}(G)\) holds for all \(a \in A\) and \(b \in B\), in particular, \eqref{e3.16} is also valid for every \(a_1 \in A^{i_1}\) and \(b_1 \in B^{j_1}\).

\(\ref{t3.6:s2} \Rightarrow \ref{t3.6:s1}\). Let \(\mathbf{G}(\mathbf{A}, \mathbf{B})\) be a complete bipartite graph. The equalities
\[
A \times B = \left(\bigcup_{i\in I} A^i\right) \times \left(\bigcup_{j\in J} B^j\right) = \bigcup_{i\in I, j\in J} (A^i \times B^j)
\]
imply that \(G\) is path-complete iff the inclusion
\begin{equation}\label{t3.6:e1}
A^i \times B^j \subseteq \mathcal{B}_{path}(G)
\end{equation}
holds for all \(i\in I\) and \(j\in J\).

Let us consider arbitrary \(i_1 \in I\) and \(j_1 \in J\). Since \(\mathbf{G}(\mathbf{A}, \mathbf{B})\) is complete bipartite, we can find \(a_1 \in A^{i}\) and \(b_1 \in B^{j}\) such that \((a_1, b_1) \in \mathcal{B}_{path}(G)\). Now inclusion \eqref{t3.6:e1} follows from Theorem~\ref{t3.4} with \(A^i = A_1\) and \(B^j = B_1\).
\end{proof}

\begin{corollary}\label{c2.9}
Let \(G = G(A, B)\) be a path-bipartite graph of the sets \(A\) and \(B\). If at least one from the sets \(A\), \(B\) has exactly one point,
\begin{equation}\label{c2.9:e1}
\min\{|A|, |B|\} \leqslant 1,
\end{equation}
then the following statements are equivalent:
\begin{enumerate}
\item\label{c2.9:s1} \(G\) is connected.
\item\label{c2.9:s2} \(G\) is path-complete.
\end{enumerate}
\end{corollary}

This corollary follows from Theorems~\ref{t3.4} and \ref{t3.6}, but, for simplicity, we give below a direct elementary proof.

\begin{proof}[Proof of Corollary~\ref{c2.9}]
The implication \(\ref{c2.9:s2} \Rightarrow \ref{c2.9:s1}\) is evidently valid. Let us prove the validity of \(\ref{c2.9:s1} \Rightarrow \ref{c2.9:s2}\).

Let \(G\) be connected. Without loss of generality, suppose that \(A = \{a_1\}\), where \(a_1\) is the unique point of \(A\). Then \(B = V(G) \setminus \{a_1\}\) holds by Proposition~\ref{p3.2}. Hence, every path joining \(a_1\) with an arbitrary \(b \in B\) is a be-path of \(A\) and \(B\) by Definition~\ref{d3.6}. Now \eqref{e3.2} implies that \(G\) is path-complete.
\end{proof}

\begin{remark}\label{r2.10}
Example~\ref{ex3.7} shows the number \(1\) is the best possible integer number for inequality~\eqref{c2.9:e1}.
\end{remark}

\begin{theorem}\label{t3.9}
Let \(G\) be a graph, \(A\) and \(B\) are disjoint nonempty subsets of \(V(G)\), and \(\{G_i \colon i \in I\}\) be the set of all connected components of \(G\), and let \(\mathbf{P}_{A, B}\) be the set of all be-paths \(P\) of \(A\) and \(B\) which are subgraphs of \(G\), \(P \subseteq G\). Then the following statements are equivalent:
\begin{enumerate}
\item \label{t3.9:s1} The equality
\begin{equation}\label{t3.9:e1}
G = \bigcup_{P \in \mathbf{P}_{A, B}} P
\end{equation}
holds, i.e., \(G\) is a path-bipartite graph of \(A\) and \(B\).
\item \label{t3.9:s2} We have
\begin{equation}\label{t3.9:e2}
V(G) = A \cup B
\end{equation}
and
\begin{equation}\label{t3.9:e3}
A \cap V(G_i) \neq \varnothing \neq B \cap V(G_i)
\end{equation}
for every \(i \in I\).
\end{enumerate}
\end{theorem}

\begin{proof}
\(\ref{t3.9:s1} \Rightarrow \ref{t3.9:s2}\). Equality \eqref{t3.9:e2} follows from Proposition~\ref{p3.2}. To prove \eqref{t3.9:e3}, suppose contrary that there is \(i_0 \in I\) such that
\[
A \cap V(G_i) = \varnothing \text{ or } B \cap V(G_i) = \varnothing.
\]
Without loss of generality, we can assume that \(B \cap V(G_{i_0}) = \varnothing\) holds. Then, using the last equality and \eqref{t3.9:e2}, we obtain
\begin{equation}\label{t3.9:e6}
V(G_{i_0}) \subseteq A.
\end{equation}
Equality \eqref{t3.9:e1} implies that there is a be-path \(P^{0} \in \mathbf{P}_{A, B}\) such that \(a_{i_0} \in V(P^{0})\). The be-path \(P^{0}\) is a connected subgraph of \(G\). Consequently, the inclusion
\begin{equation}\label{t3.9:e7}
P^{0} \subseteq G_{i_0}
\end{equation}
holds by Lemma~\ref{l2.11}. In particular, from \eqref{t3.9:e6} and \eqref{t3.9:e7} it follows that
\[
V(P^{0}) \subseteq A,
\]
contrary to Definition~\ref{d3.6}. Thus, \eqref{t3.9:e3} holds for every \(i \in I\).

\(\ref{t3.9:s2} \Rightarrow \ref{t3.9:s1}\). Let \ref{t3.9:s2} hold. We must show that equality \eqref{t3.9:e1} is valid. Let us do it.

First of all, we note that
\[
G \supseteq \bigcup_{P \in \mathbf{P}_{A, B}} P
\]
holds. Hence, \eqref{t3.9:e1} holds iff
\[
G \subseteq \bigcup_{P \in \mathbf{P}_{A, B}} P.
\]
The last inclusion means that
\begin{align}\label{t3.9:e8}
E(G) \subseteq \bigcup_{P \in \mathbf{P}_{A, B}} E(P),\\
\label{t3.9:e9}
V(G) \subseteq \bigcup_{P \in \mathbf{P}_{A, B}} V(P).
\end{align}
Using~\eqref{t3.9:e3}, we see that every connected component \(G_i\) of \(G\) contains at least two distinct vertices and, consequently,
\[
V(G_i) = \bigcup_{\{u, v\} \in E(G_i)} \{u, v\}
\]
holds. The last equality and Lemma~\ref{l2.8} give us
\[
V(G) = \bigcup_{i \in I} V(G_i) = \bigcup_{\{u, v\} \in E(G)} \{u, v\}.
\]
Hence, \eqref{t3.9:e9} follows from \eqref{t3.9:e8}. To prove \eqref{t3.9:e8} it suffices to show that every edge of \(G_i\), \(i \in I\), is also an edge of a be-path \(P \in \mathbf{P}_{A, B}\), i.e.,
\begin{equation}\label{t3.9:e10}
E(G_i) \subseteq \bigcup_{P \in \mathbf{P}_{A, B}} E(P)
\end{equation}
holds for every \(i \in I\).

Let \(i_1\) be an arbitrary index of the family \(I\) and let
\begin{equation}\label{t3.9:e11}
\{x, y\} \in E(G_{i_1}).
\end{equation}
We must find \(P \in \mathbf{P}_{A, B}\) such that
\begin{equation}\label{t3.9:e12}
\{x, y\} \in E(P).
\end{equation}
If we have
\begin{equation}\label{t3.9:e13}
A \cap \{x, y\} \neq \varnothing \neq B \cap \{x, y\},
\end{equation}
then the path \(P_{i_1}\) defined by
\begin{equation}\label{t3.9:e14}
V(P_{i_1}) = \{x, y\}, \quad E(P_{i_1}) = \{\{x, y\}\}
\end{equation}
belongs to \(\mathbf{P}_{A, B}\).

If \eqref{t3.9:e13} is not satisfied, then, without loss of generality, we can assume that
\[
\{x, y\} \subseteq V(G_1) \cap B.
\]
Let us denote by \(P^{1} = (v_1, \ldots, v_n)\) the shortest part joining the set \(\{x, y\}\) with the set \(V(G_1) \cap A\) in \(G_1\). Then we have either \(v_1 = x\), \(v_n \in A\) and \(v_i \in B \setminus \{x, y\}\) for every \(i \in \{2, \ldots, n-1\}\) or \(v_1 = y\), \(v_n \in A\) and \(v_i \in B \setminus \{x, y\}\) for every \(i \in \{2, \ldots, n-1\}\). In each of the cases, it is easy to prove that
\[
P^{1} \cup P_{i_1} \in \mathbf{P}_{A, B},
\]
where \(P_{i_1}\) is a path defined by \eqref{t3.9:e14}.
\end{proof}

Theorem~\ref{t3.9} implies the following corollaries.

\begin{corollary}\label{c3.10}
The following statements are equivalent for every graph \(G\):
\begin{enumerate}
\item\label{c3.10:s1} There are disjoint nonempty subsets \(A\) and \(B\) of the vertex set \(V(G)\) such that \(G\) is a path-bipartite graph of the sets \(A\) and \(B\).
\item\label{c3.10:s2} The equality \(G = G'\) holds.
\end{enumerate}
\end{corollary}

\begin{corollary}\label{c2.13}
The following statements are equivalent for every graph \(G\):
\begin{enumerate}
\item\label{c2.13:s1} \(G\) is a path-bipartite graph of sets \(A\) and \(B\) whenever \(A\) and \(B\) are disjoint nonempty subsets of \(V(G)\) such that \(V(G) = A \cup B\).
\item\label{c2.13:s2} \(G\) is connected and the equality \(G' = G\) holds.
\end{enumerate}
\end{corollary}

\begin{remark}\label{r3.12}
For connected graphs \(G\) the equality \(G' = G\) holds if and only if \(|V(G)| \geqslant 2\).
\end{remark}

\begin{corollary}\label{c3.7}
Let \(G = G(A, B)\) be a bipartite graph with parts \(A\) and \(B\). Then the following statements are equivalent:
\begin{enumerate}
\item\label{c3.7:s1} \(G\) is a path-complete path-bipartite graph of \(A\) and \(B\).
\item\label{c3.7:s2} \(G\) is a complete bipartite graph.
\end{enumerate}
\end{corollary}

\begin{proof}
\(\ref{c3.7:s1} \Rightarrow \ref{c3.7:s2}\). Let \ref{c3.7:s1} hold. To prove the validity of \ref{c3.7:s2} it suffices to note that every connected component of induced graphs \(G[A]\) and \(G[B]\) is a graph with one vertex and to use Theorem~\ref{t3.6}.

\(\ref{c3.7:s2} \Rightarrow \ref{c3.7:s1}\). Let \ref{c3.7:s2} hold. Then the equality \(G' = G\) is valid. Consequently, \(G\) is path-bipartite by Corollary~\ref{c3.10}. Now \ref{c3.7:s1} follows from Theorem~\ref{t3.6}.
\end{proof}

\begin{example}\label{ex3.1}
Write
\begin{align*}
x_{1} & =(1,0,0,0), & x_{2} & =(0,1,0,0), & x_{3} & = (0,0,1,0), & x_{4} &=(0,0,0,1),\\
x_{5} & =(1,0,1,0), & x_{6} & =(1,1,0,0), & x_{7} & = (1,0,0,1), & x_{8} & =(0,0,0,0),\\
x_{9} & =(0,1,1,0), & x_{10} & =(0,1,0,1),& x_{11} & =(0,0,1,1), & x_{12} & =(0,1,1,1),\\
x_{13}& =(1,1,0,1), & x_{14} & =(1,1,1,1),& x_{15} & =(1,0,1,1), & x_{16} & =(1,1,1,1).
\end{align*}
Let us define a graph $G=(V,E)$ as
\begin{align*}
V(G) & := \{x_{1},x_{2},x_{3},x_{4},x_{5},x_{6},x_{7},x_{8}, x_{9},x_{10},x_{11},x_{12}, x_{13},x_{14},x_{15},x_{16}\}  \text{ and}\\
E(G) &  := \{ \{x_{1},x_{5}\},
\; \{x_{1},x_{6}\}, \;\{x_{1},x_{7}\}, \; \{x_{1},x_{8}\},\; \{x_{2},x_{6}\},\; \{x_{2},x_{9}\}, \;\{x_{2},x_{10}\},\\
& \hspace*{6mm} \{x_{3},x_{5}\}, \,\{x_{3},x_{8}\}, \;\{x_{3},x_{9}\}, \;\{x_{3},x_{11}\}, \; \{x_{4},x_{7}\}, \;\{x_{4},x_{8}\}, \;\{x_{4},x_{10}\},\\
& \hspace*{6mm}  \{x_{5},x_{14}\}, \;\{x_{5},x_{15}\}, \;\{x_{6},x_{13}\}, \;\{x_{6},x_{14}\}, \;\{x_{7},x_{15}\}, \; \{x_{9},x_{12}\},\\
& \hspace*{6mm}  \{x_{10},x_{12}\}, \;\{x_{11},x_{12}\}, \;\{x_{13},x_{16}\}, \;\{x_{14},x_{16}\}, \;\{x_{15},x_{16}\} \}.
\end{align*}
Then, by Corollary~\ref{c2.13}, $G$ is a path-bipartite graph of the sets
\begin{equation}\label{e3.1:e1}
A = \{ x_{1},x_{2},x_{3},x_{4}, x_{9},x_{10},x_{11},x_{12}\} \text{ and }
B = \{x_{5},x_{6},x_{7},x_{8}, x_{13},x_{14},x_{15},x_{16}\}
\end{equation}  	
because \(G\) is connected and \(A \cup B = V(G)\).
\end{example}

\section{Path-proximinal graphs}
\label{sec3}

The first theorem of the section describes the geometry of proximinal pairs \((A, B)\) in semimetric spaces \((X, d)\) for which the graphs \(G\), \(V(G) = X\) and \(E(G)\) defined by \eqref{d3.10:e1}, are path-bipartite. In what follows \(A_0\) and \(B_0\) are subsets of \(A\) and \(B\) defined by \eqref{d1.2:e1} and, respectively, by \eqref{d1.2:e2}.

\begin{theorem}\label{t3.5}
Let \((X, d)\) be a semimetric space and let \(A\), \(B\) be disjoint proximinal subsets of \((X, d)\) such that
\begin{equation}\label{t3.5:e1}
X = A \cup B.
\end{equation}
Let us consider a graph \(G\) such that \(V(G) = X\) and the equivalence
\[
(\{x,y\} \in E(G)) \Leftrightarrow (d(x,y) \leqslant \dist(A,B))
\]
is valid for all distinct \(x\), \(y \in X\). Then \(G\) is path-bipartite for \(A\) and \(B\) if and only if the following conditions are fulfilled:
\begin{enumerate}
\item\label{t3.5:c1} For every \(a \in A \setminus A_0\) there are a best proximity pair \((a^*, b^*) \in A \times B\) and a finite path \((a_1, \ldots, a_n) \subseteq G[A]\) such that \(a_1 = a\), \(a_n = a^*\).
\item\label{t3.5:c2} For every \(b \in B \setminus B_0\) there are a best proximity pair \((a^*, b^*) \in A \times B\) and a finite path \((b_1, \ldots, b_n) \subseteq G[B]\) such that \(b_1 = b\), \(b_n = b^*\).
\end{enumerate}
\end{theorem}

\begin{proof}
Let \(G\) be a path-bipartite graph. We must show that conditions \ref{t3.5:c1}--\ref{t3.5:c2} are valid.

\ref{t3.5:c1}. Let \(a\) be a point of \(A \setminus A_0\). By Definition~\ref{d3.6}, there is a be-path \(P_{A, B} \subseteq G\) such that \(a \in V(P_{A, B})\). Using this definition again we can find \(a^* \in A\) such that \(a^* \in V(P_{A, B})\). The be-path \(P_{A, B}\) is a connected graph, consequently, there is a path \(P \subseteq P_{A, B}\) joining \(a\) and \(a^*\) in \(P_{A, B}\). To complete the proof of \ref{t3.5:c1} it suffices to show that \(V(P) \subseteq A\). Suppose contrary that there is a point \(b^0 \in B\) such that \(b^0 \in V(P)\). Since \(P\) is a path joining \(a\) and \(a^*\), there are paths \(P_1 \subseteq P\) and \(P_2 \subseteq P\) such that
\begin{equation}\label{t3.5:e2}
a, b^0 \in V(P_1), \quad b^0, a^* \in V(P_2), \quad P = P_1 \cup P_2 \quad\text{and}\quad V(P_1) \cap V(P_2) = \{b^0\}.
\end{equation}
Using \eqref{t3.5:e2}, \(V(P) \subseteq A \cup B\) and \(a_1\), \(a^* \in A\) and \(b^0 \in B\), we can find \(\{a^1, b^1\} \in E(P_1)\) and \(\{a^2, b^2\} \in E(P_2)\) such that
\begin{equation}\label{t3.5:e3}
a^1, a^2 \in A \quad \text{and}\quad b_1, b_2 \in B.
\end{equation}
The last equality in \eqref{t3.5:e2} implies that \(\{a^1, b^1\}\) and \(\{a^2, b^2\}\) are different edges of \(P_{A, B}\) that together with \eqref{t3.5:e3} contradicts the definition of be-paths.

\ref{t3.5:c2}. The validity of \ref{t3.5:c2} can be proved similarly.

Suppose now that \ref{t3.5:c1} and \ref{t3.5:c2} hold. To prove that \(G\) is a path-bipartite graph of sets \(A\) and \(B\) we consider an arbitrary connected component \(G_j\) of \(G\) and an arbitrary point \(p \in V(G_j)\). Equality~\eqref{t3.5:e1} implies that \(p \in A\) or \(p \in B\). Without loss of generality we assume that \(p \in A\). If \(p \in A_0\) holds, then, by \eqref{d1.2:e1}, there is \(q \in B\) such that \(d(p, q) = \dist(A, B)\). Hence, \(\{p, q\} \in E(G)\) by definition of \(G\) and, consequently, \(q \in V(G_j)\) by Lemma~\ref{l2.11}. If \(p \in A \setminus A_0\), then, by condition~\ref{t3.5:c1}, there is a path
\begin{equation}\label{t3.5:e4}
P \subseteq G
\end{equation}
joining \(p\) with a point \(p_0 \in A_0\). Since \(G_j\) is a connected component of \(G\), inclusion~\eqref{t3.5:e4} implies \(P \subseteq G_j\) by Lemma~\ref{l2.9}. Thus, the point \(p_0 \in A_0\) also belongs to \(V(G_j)\). Now, arguing as above, we can find \(q \in B \cap V(G_j)\). Hence, \(G\) is path-bipartite by Theorem~\ref{t3.9}.
\end{proof}

\begin{corollary}\label{c3.2}
Let \(G\) be a path-proximinal graph for sets \(A\) and \(B\) with respect to a semimetric \(d\) on the set \(X = A \cup B\). Then the inequality
\begin{equation}\label{c3.2:e1}
\dist(A, B) > 0
\end{equation}
holds and there are points \(a_0 \in A\) and \(b_0 \in B\) such that
\begin{equation}\label{c3.2:e2}
d(a_0, b_0) = \dist(A, B).
\end{equation}
\end{corollary}

\begin{proof}
By Definition~\ref{d3.7}, every path-proximinal graph for sets \(A\) and \(B\) are path-bipartite graph of these sets. Hence, by Theorem~\ref{t3.5}, there is a best proximity pair \((a_0, b_0) \in A \times B\) which satisfies~\eqref{c3.2:e2}. Now \eqref{c3.2:e1} follows from \eqref{c3.2:e2} and the definition of semimetrics.
\end{proof}

\begin{theorem}\label{t3.7}
For every path-bipartite graph \(G\) of sets \(A\) and \(B\), there is a metric \(d\) on \(X:= V(G)\) such that \(G\) is path-proximinal for \(A\) and \(B\) with respect to the metric \(d\).
\end{theorem}

\begin{proof}
Let \(G\) be  path-bipartite graph of sets \(A\) and \(B\). We must find a metric \(d \colon X \times X \to [0, \infty[\), \(X = A \cup B\), such that \(A\) and \(B\) are proximinal subsets of \((X, d)\) and \eqref{d3.10:e1} holds for all distinct \(x\), \(y \in X\). Let us define a function \(d \colon X \times X \to [0, \infty[\) as
\begin{equation}\label{t3.7:e1}
d(x, y) = \begin{cases}
0 & \text{if } x = y\\
1 & \text{if } \{x, y\} \in E(G)\\
2 & \text{if } x \neq y \text{ and } \{x, y\} \notin E(G).
\end{cases}
\end{equation}
Then \(d\) evidently is a metric on \(X\).

The sets \(A\) and \(B\) are nonvoid by Definition~\ref{d3.6}. Since the set \(\{d(x, y) \colon x, y \in X\}\) is finite, every nonempty subset of \(X\) is proximinal. Hence, \(A\) and \(B\) are proximinal in \((X, d)\).

We claim that
\begin{equation}\label{t3.7:e2}
\dist(A, B) = 1.
\end{equation}
Indeed, since \(A\) and \(B\) disjoint and nonvoid, \eqref{t3.7:e1} and \eqref{e1.2} give us the inequality
\begin{equation}\label{t3.7:e3}
\dist(A, B) \geqslant 1.
\end{equation}
Since \(G\) is path-bipartite of \(A\) and \(B\), there is a be-path \(P_{A, B} \subseteq G\). Consequently, by Definition~\ref{d3.6}, there is \(\{a_0, b_0\} \in E(P_{A, B})\) such that
\begin{equation}\label{t3.7:e4}
A \cap \{a_0, b_0\} \neq \varnothing \neq B \cap \{a_0, b_0\}.
\end{equation}
Using the equality \(A \cap B = \varnothing\) and \eqref{t3.7:e4} we may suppose, without loss of generality, that
\begin{equation}\label{t3.7:e5}
a_0 \in A \text{ and } b_0 \in B.
\end{equation}
From \(\{a_0, b_0\} \in E(P_{A, B})\) and \(P_{A, B} \subseteq G\) it follows that
\begin{equation}\label{t3.7:e6}
\{a_0, b_0\} \in E(G).
\end{equation}
Now \eqref{t3.7:e1} and \eqref{t3.7:e6} imply \(d(a_0, b_0) = 1\) and, consequently,
\begin{equation}\label{t3.7:e7}
\dist(A, B) \leqslant 1.
\end{equation}
holds by \eqref{t3.7:e5}. Now \eqref{t3.7:e2} follows from \eqref{t3.7:e3} and \eqref{t3.7:e7}. To complete the proof it suffices to show that \eqref{d3.10:e1} holds for all distinct \(x\), \(y \in X\). Indeed, using \eqref{t3.7:e3}, we obtain the equivalence
\[
(\{x, y\} \in E(G)) \Leftrightarrow (d(x, y) = \dist(A, B))
\]
for all \(x\), \(y \in X\). Now the equality \(\dist (A, B) = 1\) and \eqref{t3.7:e1} imply the equivalence
\[
(d(x, y) \leqslant \dist(A, B)) \Leftrightarrow (d(x, y) = \dist(A, B)),
\]
for any pair of distinct \(x\), \(y \in X\), .
\end{proof}

\begin{example}\label{ex3.2}
Let \(X\) be the set of all sequences \(\widetilde{q} = (\eta_1, \eta_2, \eta_3, \eta_4)\), where each \(\eta_i \in \{0, 1\}\) and let \(A\), \(B\subseteq X\) and the graph \(G\) be defined as in Example~\ref{ex3.1}. Let us denote by \(d(\widetilde{p}, \widetilde{q})\) the Hamming distance between \(\widetilde{p}\), \(\widetilde{q} \in X\),
\begin{equation*}
d(\widetilde{p}, \widetilde{q}) = \sum_{i=1}^4 |\mu_i - \eta_i|.
\end{equation*}
Then \((X, d)\) is a metric space, $A$ and $B$	are disjoint proximinal subsets of \((X, d)\), the equality \(\dist(A, B) = 1\) holds, vertices \(\widetilde{p}\) and \(\widetilde{q}\) of \(G\) are adjacent iff \(d(\widetilde{p}, \widetilde{q}) = 1\), and
\begin{align*}
\mathcal{B}_{path}(G) & = \{
(x_{1},x_{5}), (x_{1},x_{6}), (x_{1},x_{7}), (x_{1},x_{8}), (x_{1},x_{13}), (x_{1},x_{14}), (x_{1},x_{15}), \\
& \hspace*{6mm} (x_{1},x_{16}), (x_{2},x_{6}), (x_{2},x_{13}), (x_{2},x_{16}), (x_{3},x_{5}), (x_{3},x_{8}), (x_{3},x_{14}), \\
& \hspace*{6mm} (x_{3},x_{15}), (x_{3},x_{16}),(x_{4},x_{7}), (x_{4},x_{8}), (x_{4},x_{13}), (x_{4},x_{15}), (x_{4},x_{16}), \\
& \hspace*{6mm} (x_{9},x_{5}), (x_{9},x_{6}), (x_{9},x_{8}),(x_{9},x_{13}), (x_{9},x_{14}), (x_{9},x_{15}), (x_{9},x_{16}), \\
& \hspace*{6mm} (x_{10},x_{6}), (x_{10},x_{8}), (x_{10},x_{13}), (x_{10},x_{14}),(x_{10},x_{16}), (x_{11},x_{5}), (x_{11},x_{8}),\\
& \hspace*{6mm}  (x_{11},x_{14}),(x_{11},x_{15}), (x_{11},x_{16}), (x_{12},x_{5}), (x_{12},x_{6}), (x_{12},x_{7}), (x_{12},x_{8}), \\
& \hspace*{6mm} (x_{12},x_{13}), (x_{12},x_{14}), (x_{12},x_{15}), (x_{12},x_{16}) \}.
\end{align*}
The element $(x_{4},x_{15})$ of $\mathcal{B}_{path}(G)$ admits the following be-paths
\begin{align*}
&(x_{4},x_{7},x_{15});\ (x_{4},x_{7},x_{13},x_{6},x_{14},x_{16},x_{15});\  (x_{4},x_{7},x_{13},x_{16},x_{15}) \text{ and }\\
&(x_{4},x_{7},x_{13},x_{16},x_{14},x_{6},x_{13},x_{16},x_{15}).
\end{align*}
The graph \(G\) is path-proximinal for \(A\) and \(B\) w.r.t. the Hamming distance \(d\).
\end{example}

\begin{example}\label{ex3.16}
Let \(G\) be a graph with \(V(G) = A\) such that \(\{\tilde{x}, \tilde{y}\} \in E(G)\) iff \(d(\tilde{x}, \tilde{y}) = 1\), where \(A\) and \(d\) are defined as in Example~\ref{ex3.2}. The point \(x_1\) is an isolated vertex of \(G\) (see Figure~\ref{fig2}). Hence, the bipartite graph \(G\) is not path-proximinal by Corollary~\ref{c2.13} and Theorem~\ref{t3.7}.
\end{example}

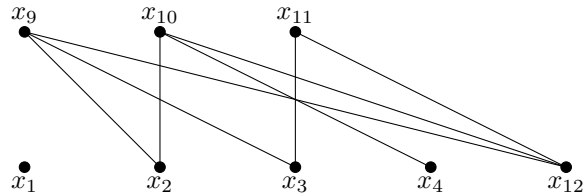
\begin{figure}[ht]
\begin{tikzpicture}[scale=1,
arrow/.style = {-{Stealth[length=5pt]}, shorten >=2pt}]

\def\xx{1.8cm}
\def\yy{0.9cm}

\coordinate [label=below:{$x_{1}$}] (A1) at (0*\xx, -\yy);
\coordinate [label=below:{$x_{2}$}] (A2) at (1*\xx, -\yy);
\coordinate [label=below:{$x_{3}$}] (A3) at (2*\xx, -\yy);
\coordinate [label=below:{$x_{4}$}] (A4) at (3*\xx, -\yy);
\coordinate [label=below:{$x_{12}$}] (A8) at (4*\xx, -\yy);
\coordinate [label=above:{$x_{9}$}] (A5) at (0*\xx, \yy);
\coordinate [label=above:{$x_{10}$}] (A6) at (1*\xx, \yy);
\coordinate [label=above:{$x_{11}$}] (A7) at (2*\xx, \yy);

\draw [fill, black] (A1) circle (2pt);
\draw [fill, black] (A2) circle (2pt);
\draw [fill, black] (A3) circle (2pt);
\draw [fill, black] (A4) circle (2pt);
\draw [fill, black] (A5) circle (2pt);
\draw [fill, black] (A6) circle (2pt);
\draw [fill, black] (A7) circle (2pt);
\draw [fill, black] (A8) circle (2pt);
\draw (A2) -- (A5) -- (A3) -- (A7);
\draw (A4) -- (A6) -- (A2);
\draw (A5) -- (A8) -- (A7);
\draw (A8) -- (A6);
\end{tikzpicture}
\caption{The point \(x_1\) is an isolated vertex of \(G\).}
\label{fig2}
\end{figure}

\begin{example}\label{ex3.12}
Let us consider a metric space $(\CC, d)$, where \(\CC\) is the set of all complex numbers $z=x+iy$. Suppose that for arbitrary $z_{1} = x_{1}+iy_{1}$, $z_{2} = x_{2}+iy_{2}$ we have
$$
d(z_{1},z_{2})=
\begin{cases}
\frac{1}{2}|[x_{1}]-[x_{2}]| + | [y_{1}]-[y_{2}]| +1 & \text{ if } z_{1}\neq z_{2}\\
0 & \text{ if } z_{1}= z_{2},
\end{cases}
$$
where $[x_i]$ \(([y_i])\) is the integer part of $x_i$ \((y_i)\).

Write
\[
A := \{n+ih\colon n\in \nit^{*} \text{ and } h=0\} \text{ and }
B := \{m+ik\colon m\in \nit \text{ and } k\in \nit^{*}\},
\]
where \(\nit^*\) is the set of all naturals without zero. Then the equalities
\begin{align*}
\dist(A,B) & = \inf\{ d(a,b)\colon (a,b)\in A\times B\} \\
& = \inf\left\{\frac{1}{2}|n-m|+ |k|+1\colon (n,m,k)\in \nit^{*}\times \nit \times \nit^{*}\right\} = 2
\end{align*}
hold.

Since every nonempty subset of the set
\[
\{d(z_1, z_2) \colon z_1, z_2 \in \CC\}
\]
has the smallest element, each nonempty \(S \subseteq \CC\) is proximinal subset of the metric space \((\CC, d)\). Hence, \((A, B)\) is a proximinal pair for \((\CC, d)\).

Let us define a graph $G$ such that $V(G)= A \cup B$ and, for \(z_1\), \(z_2 \in A \cup B\), \(\{z_1, z_2\} \in E(G)\) iff \(0 < d(z_1, z_2) \leqslant 2\). It is easy to show that:
\begin{itemize}
\item For all $n\in \nit^*$, $d(n,n+1)= \frac{3}{2} < \dist(A,B)$;
\item For $m\in \nit^*$, $d(m,m+i)= 2 = \dist(A,B)$;
\item For all $(m,k)\in \nit \times \nit^*$, $d(m+ik,m+i(k+1))=  2 = \dist(A,B)$.
\end{itemize}	
Let us consider $n\in \nit^*$ and $(m,k)\in \nit \times \nit^*$. Suppose $n\leqslant m$. Then we get the be-path \(P\)
\begin{align*}
V(P) & = \{n,n+1,\cdots,m-1,m,m+i,\cdots,m+i(k-1),m+ik\}, \\
E(P) & = \{\{n,n+1\}, \cdots,\{m-1,m\},\{m,m+i\}, \cdots, \{m+i(k-1),m+ik\}\}.
\end{align*}
If $m< n$, then we get the following be-path \(P_1\):
\begin{align*}
V(P_1) & = \{n,n+1,\cdots,m+1,m,m+i,\cdots,m+i(k-1),m+ik\}, \\
E(P_1) & = \{\{n,n+1\}, \cdots,\{m+1,m\},\{m,m+i\},\cdots, \{m+i(k-1),m+ik\}\}.
\end{align*}
Thus, $\mathcal{B}_{path}(G)= \{(n,m+ik)\colon (n,m,k)\in \nit^{*} \times \nit \times \nit^{*}\} = A\times B$ and, consequently, the graph \(G\) is path-complete and path-proximinal for $A$ and $B$ with respect to $d$.
\end{example}

The next theorem follows from Corollary~\ref{c3.10} and Theorem~\ref{t3.7}.

\begin{theorem}\label{t3.16}
Let $G$ be a graph. Then the following statements are equivalent:
\begin{enumerate}
\item\label{t3.16:s1} $G$ does not contain any isolated vertices.
\item\label{t3.16:s2} $G$ is a path-proximinal graph for a semimetric space.
\item\label{t3.16:s3} \(G\) is a path-proximinal graph for a metric space.
\end{enumerate}	
\end{theorem}

\begin{proposition}\label{p3.22}
Let a bipartite graph \(G\) with fixed parts \(A\) and \(B\) be proximinal for a semimetric space \((X,d)\), where \(X = V(G) = A \cup B\). Then $G$ is a path-proximinal graph if and only if the equalities
\begin{equation}\label{p3.22:e1}
A_0 = A \quad \text{and} \quad B_0 = B
\end{equation}
holds, where the sets \(A_0\) and \(B_0\) are defined by \eqref{d1.2:e1} and \eqref{d1.2:e2}, respectively.
\end{proposition}

\begin{proof}
By Theorem~\ref{t3.16}, the graph \(G\) is path-proximinal iff \(G' = G\) holds. Since \(G\) is proximinal for semimetric space \((X, d)\), Definition~\ref{d2.14} and Definition~\ref{d2.2} imply that the equality \(G' = G\) holds if and only if \eqref{p3.22:e1} is valid.
\end{proof}

The conditions \(A_0 = A\) and \(B_0 = B\) guarantee the existence of a semimetric \(\rho \colon X \times X \to [0, \infty[\) such that \(G\) is path-proximinal w.r.t. \(\rho\). In the next proposition we describe geometric properties of the space \((X, d)\) under which the equality \(d = \rho\) is possible.

\begin{proposition}\label{p3.9}
Let \(G = G(A, B)\) be a proximinal graph for a semimetric space \((X, d)\) with \(X = A \cup B\) and let \(G = G'\) hold. Then the following conditions are equivalent:
\begin{enumerate}
\item\label{p3.9:c1} \(G\) is path-proximinal for \(A\) and \(B\) w.r.t. the semimetric \(d\).
\item\label{p3.9:c2} The inequality
\begin{equation}\label{p3.9:e1}
d(x, y) > \dist(A, B)
\end{equation}
holds whenever \(x \neq y\) and \(x\), \(y \in A\) or \(x\), \(y \in B\).
\end{enumerate}
\end{proposition}

\begin{proof}
\(\ref{p3.9:c1} \Rightarrow \ref{p3.9:c2}\). Let \ref{p3.9:c1} hold. Definition~\ref{d2.4} implies that
\begin{equation}\label{p3.9:e2}
\{x, y\} \notin E(G)
\end{equation}
whenever \(x\), \(y \in A\) or \(x\), \(y \in B\). By Definition~\ref{d3.7}, \eqref{p3.9:e2} holds for distinct points \(x\), \(y \in X\) if and only if we have \eqref{p3.9:e1} for these points. Condition~\ref{p3.9:c2} follows.

The validity of \(\ref{p3.9:c2} \Rightarrow \ref{p3.9:c1}\) can be proved similarly.
\end{proof}

Proposition~\ref{p3.9} admits the following ``ultrametric modification''.

\begin{theorem}\label{t3.10}
The following conditions are equivalent for every graph \(G\):
\begin{enumerate}
\item\label{t3.10:s1} For every vertex \(v\) of \(G\) there is a unique vertex \(u\) of \(G\) such that \(u\) and \(v\) are adjacent.
\item\label{t3.10:s2} There are an ultrametric \(d\) on the set \(X = V(G)\) and disjoint proximinal subsets \(A\), \(B\) of \(X\) such that \(X = A \cup B\), and \(G\) is bipartite with the parts \(A\) and \(B\), and path-proximinal w.r.t. the ultrametric \(d\).
\end{enumerate}
\end{theorem}

\begin{proof}
\(\ref{t3.10:s1} \Rightarrow \ref{t3.10:s2}\). Let \(G\) satisfy condition~\ref{t3.10:s1}. Then \(G\) has no isolated vertices and, hence,
\begin{equation}\label{t3.10:e1}
V(G) = \bigcup_{\{x, y\} \in E(G)} \{x, y\}
\end{equation}
holds. Condition~\ref{t3.10:s1} implies also the equality
\begin{equation}\label{t3.10:e2}
\{x, y\} \cap \{u, v\} = \varnothing
\end{equation}
whenever \(\{x, y\}\) and \(\{u, v\}\) are different edges of \(G\). Now using \eqref{t3.10:e1} and \eqref{t3.10:e2} and the axiom of choice we can find disjoint subsets \(A\) and \(B\) of \(V(G)\) such that \(V(G) = A \cup B\) and
\[
\{u, v\} \cap A \neq \varnothing \neq \{u, v\} \cap B
\]
for every \(\{u, v\} \in E(G)\). Write
\begin{equation}\label{t3.10:e3}
X = A \cup B
\end{equation}
and define a semimetric \(d \colon X \times X \to [0, \infty[\) as
\begin{equation}\label{t3.10:e4}
d(x, y) = \begin{cases}
0 & \text{if } x = y\\
1 & \text{if } \{x, y\} \in E(G)\\
2 & \text{otherwise}.
\end{cases}
\end{equation}
We claim that \(d\) is an ultrametric on \(X\). By definition, \(d\) is an ultrametric if
\begin{equation}\label{t3.10:e5}
d(x, y) \leqslant \max\{d(x, z), d(z, y)\}
\end{equation}
holds for all \(x\), \(y\), \(z \in X\). It is easy to see that \eqref{t3.10:e5} holds for arbitrary semimetric \(d\) on \(X\) if \(|\{x, y, z\}| \leqslant 2\). Let \(x\), \(y\), \(z\) be pairwise distinct points of \(X\). If \eqref{t3.10:e5} is false,
\[
d(x, y) > \max\{d(x, z), d(z, y)\},
\]
then, using \eqref{t3.10:e4}, we see that \(d(x, y) = 2\) and
\begin{equation}\label{t3.10:e6}
d(x, z) = d(z, y) = 1.
\end{equation}
It follows from \eqref{t3.10:e4} and \eqref{t3.10:e6} that the vertex \(z\) has two different adjacent vertices \(x\) and \(y\), contrary to \ref{t3.10:s1}. Thus, \(d\) is an ultrametric on \(X\).

The sets \(A\) and \(B\) do not intersect by construction and, in addition, \eqref{t3.10:e3} holds. Since the set \(\{d(x, y) \colon x, y \in X\}\) is finite by~\eqref{t3.10:e4}, these sets are also proximinal subsets of \((X, d)\). It follows from Definition~\ref{d2.3} that \(G\) is bipartite graph with parts \(A\) and \(B\). Equality~\eqref{t3.10:e4} and the definition of the sets \(A\) and \(B\) also give us the equality \(\dist(A, B) = 1\). Now using~\eqref{t3.10:e4} and Definition~\ref{d3.7} we see that \(G\) is path-proximinal w.r.t. the ultrametric \(d\).

\(\ref{t3.10:s2} \Rightarrow \ref{t3.10:s1}\). Let \(G\) satisfy condition~\ref{t3.10:s2}. Then \(G\) is path-proximinal for an ultrametric space \((X, d)\) and, consequently, \(G\) has no isolated vertices by Theorem~\ref{t3.16}. Suppose that there is a vertex \(v\) of \(G\) such that
\begin{equation}\label{t3.10:e7}
\{v, u\}, \{v, w\} \in E(G)
\end{equation}
for some different \(u\), \(w \in V(G)\). By condition~\ref{t3.10:s2}, \(G\) is a bipartite graph with parts \(A\) and \(B\). Consequently, \eqref{t3.10:e7} implies either
\begin{equation}\label{t3.10:e8}
v \in A \text{ and } u, w \in B
\end{equation}
or
\begin{equation}\label{t3.10:e9}
u, w \in A \text{ and } v \in B.
\end{equation}
Suppose that \eqref{t3.10:e8} holds. Then using the strong triangle inequality and Definition~\ref{d2.4} we obtain
\[
d(u, w) \leqslant \max \{d(u, v), d(v, w)\} = \dist(A, B).
\]
Hence, the vertices \(u\) and \(w\) are adjacent in \(G\),
\begin{equation}\label{t3.10:e10}
\{u, w\} \in E(G)
\end{equation}
by Definition~\ref{d3.7}. Now it suffices to note that \eqref{t3.10:e8} implies \(\{u, w\} \notin E(G)\) by Definition~\ref{d2.4}, contrary to~\eqref{t3.10:e10}. Thus, \eqref{t3.10:e8} is false. Analogously, we obtain that \eqref{t3.10:e9} is also false. The proof of the validity of \(\ref{t3.10:s2} \Rightarrow \ref{t3.10:s1}\) is complete.
\end{proof}

Theorem~\ref{t3.10} implies, in particular, the following.

\begin{corollary}\label{c3.12}
A graph \(G\) is simultaneously proximinal and path-proximinal for an ultrametric space \((X, d)\) and given disjoint proximinal subsets \(A\) and \(B\) of \(X\) iff every connected component of \(G\) has exactly \(2\) vertices.
\end{corollary}

\begin{corollary}\label{c3.11}
Let \(G\) be a bipartite graph with parts \(A\) and \(B\). Suppose that there is an ultrametric \(d \colon X \times X \to [0, \infty[\) such that \(G\) is path-proximinal with respect to \(d\). Then the following conditions are equivalent:
\begin{enumerate}
\item\label{c3.11:s1} \(G\) is connected.
\item\label{c3.11:s2} \(G\) is complete.
\item\label{c3.11:s3} The induced-bipartite subgraph \(G[A, B]\) of \(G\) is complete bipartite.
\item\label{c3.11:s4} \(G\) is path-complete.
\end{enumerate}
\end{corollary}

\section*{Funding}

Oleksiy Dovgoshey was partially supported by Finnish Society of Sciences and Letters, Project ``Intrinic Metrics of Domains in Geometric Function Theory and Graphs''.

\end{document}